\documentclass{article}

\pdfoutput=1

\usepackage[T2A, T1]{fontenc}
\let\C=\undefined

\usepackage{ajf-0.3}
\usepackage{amsthm}
\usepackage{tikz}

\newcommand{\bicornplane}{
\begin{scope}[nodes={outer sep=0}]
\fill[smoked!40] (0, -2) rectangle (4, 2);
\fill[smoked!20] (0, -3/2) rectangle (4, -1);
\fill[smoked!20] (0, -1/2) rectangle (4, 0);
\fill[smoked!20] (0, 1/2) rectangle (4, 1);
\fill[smoked!20] (0, 3/2) rectangle (4, 2);
\node[left] at (0, -1) {$-\pi$};
\node[left] at (0, 0) {$0$};
\node[left] at (0, 1) {$\pi$};
\end{scope}
}

\newcommand{\twopunkplane}[4]{
\begin{scope}[nodes={outer sep=0}]
\fill[#3] (-2, -2) rectangle (2, 2);
\fill[#4] (-2, -2) rectangle (2, 0);
\fill[#2] (0, 0) circle[radius=1];
\fill[#1]
  (-1, 0) coordinate (zero)
  arc[start angle=180, radius=1, delta angle=180] coordinate (infty);
\end{scope}
}

\newcommand{\twopunks}{
\fill[black] (zero) circle[radius=0.5mm];
\fill[black] (infty) circle[radius=0.5mm];
\node[left] at (zero) {$0$};
\node[right] at (infty) {$\infty$};
}

\newcommand{\punkdsphere}[2]{
\node[circle, minimum size=#1*4cm] (0, 0) (#2) {};
\begin{scope}[line width=0.2mm, scale=#1]
\draw[black, fill] (-1.8, 0) circle [x radius=0.175mm, y radius=0.5mm];
\draw[mauve] (0, -2) arc [x radius=7mm, y radius=2cm, start angle=-90, end angle=90];
\fill[smoked!50, opacity=0.8, blend mode=screen] (0, 0) circle (2);
\draw[mauve] (0, 2) arc [x radius=7mm, y radius=2cm, start angle=90, end angle=270];
\draw[black, fill] (1.8, 0) circle [x radius=0.175mm, y radius=0.5mm];
\draw[black] (0, 0) circle (2);
\end{scope}
}

\definecolor{cd8d4cc}{RGB}{216,212,204}
\definecolor{cc837ab}{RGB}{200,55,171}

\newcommand{\inftbicorn}[2]{
\node[circle, minimum size=#1*4cm] (0, 0) (#2) {};
\begin{scope}[y=1mm, x=1mm, yscale=-#1, xscale=#1, shift={(-29.7, -29.7)}]
\begin{scope}[blend mode=screen,opacity=0.400,transparency group]% layer7
  \path[scale=0.265,fill=cd8d4cc,line cap=round,miter limit=4.00,nonzero rule]
    (111.4961,35.9062) .. controls (69.2795,35.9062) and
    (34.9515,69.6708) .. (34.0371,111.6680) .. controls (34.2016,104.7521) and
    (35.3526,98.0642) .. (37.3789,92.9883) .. controls (40.6309,84.8423) and
    (45.4981,82.4413) .. (49.7207,86.9004) .. controls (53.9433,91.3595) and
    (56.6934,101.8066) .. (56.6934,113.3867) .. controls (56.6934,121.9306) and
    (54.8922,129.8237) .. (51.9688,134.0957) .. controls (49.0453,138.3677) and
    (45.4430,138.3677) .. (42.5195,134.0957) .. controls (39.5961,129.8237) and
    (37.7949,121.9306) .. (37.7949,113.3867) .. controls (37.7949,140.3926) and
    (52.2020,165.3467) .. (75.5898,178.8496) .. controls (98.9776,192.3526) and
    (127.7938,192.3526) .. (151.1816,178.8496) .. controls (174.5694,165.3467) and
    (188.9766,140.3926) .. (188.9766,113.3867) .. controls (188.9766,70.5955) and
    (154.2873,35.9062) .. (111.4961,35.9062) -- cycle(34.0371,111.6680) ..
    controls (34.0338,111.8051) and (34.0337,111.9428) .. (34.0312,112.0801) ..
    controls (34.0335,111.9428) and (34.0341,111.8050) .. (34.0371,111.6680) --
    cycle(34.0312,112.0801) .. controls (34.0240,112.5152) and (34.0156,112.9499)
    .. (34.0156,113.3867) -- (34.0215,113.3867) .. controls (34.0215,112.9504) and
    (34.0234,112.5151) .. (34.0313,112.0801) -- cycle;

\end{scope}
\begin{scope}% layer15
  \path[shift={(0,-237.0)},draw=black,line cap=round,miter limit=4.00,nonzero
    rule,line width=0.2mm] (9.0000,267.0000)arc(180.000:270.000:3.000002 and
    7.593)arc(270.000:360.000:3.000002 and 7.593);

  \path[shift={(0,-237.0)},draw=black,line cap=round,miter limit=4.00,nonzero
    rule,line width=0.2mm] (15.0000,267.0000)arc(0.000:60.000:2.500003 and
    6.327)arc(60.000:120.000:2.500003 and 6.327)arc(120.000:180.000:2.500003 and
    6.327);

  \path[shift={(0,-237.0)},draw=cc837ab,line cap=round,miter limit=4.00,nonzero
    rule,line width=0.2mm] (30.0000,246.5044)arc(-90.000:0.000:8.000002 and
    20.248)arc(0.000:90.000:8.000002 and 20.248);

\end{scope}
\begin{scope}[blend mode=screen,opacity=0.400,transparency group]% layer9
  \path[fill=cd8d4cc,line cap=round,miter limit=4.00,nonzero rule]
    (29.5000,10.5000) .. controls (18.7304,10.5000) and
    (10.0000,19.2304) .. (10.0000,30.0001) .. controls (10.0000,26.6784) and
    (10.8954,23.9856) .. (12.0000,23.9856) .. controls (13.1046,23.9856) and
    (14.0000,26.6784) .. (14.0000,30.0002) .. controls (14.0000,31.7027) and
    (13.7141,33.2758) .. (13.2500,34.1270) .. controls (12.7859,34.9782) and
    (12.2141,34.9782) .. (11.7500,34.1270) .. controls (11.2859,33.2758) and
    (11.0000,31.7027) .. (11.0000,30.0001) .. controls (11.0000,36.7880) and
    (14.6214,43.0605) .. (20.5000,46.4545) .. controls (26.3786,49.8485) and
    (33.6214,49.8485) .. (39.5000,46.4545) .. controls (45.3786,43.0605) and
    (49.0000,36.7880) .. (49.0000,30.0000) .. controls (49.0000,19.2304) and
    (40.2696,10.5000) .. (29.5000,10.5000) -- cycle;

\end{scope}
\begin{scope}% layer14
  \path[shift={(0,-237.0)},draw=black,line cap=round,miter limit=4.00,nonzero
    rule,line width=0.2mm] (10.0000,267.0002)arc(180.000:270.000:1.999999 and
    6.015)arc(270.000:360.000:1.999999 and 6.015);

  \path[shift={(0,-237.0)},draw=black,line cap=round,miter limit=4.00,nonzero
    rule,line width=0.2mm] (14.0000,267.0002)arc(-0.000:60.000:1.500000 and
    4.765)arc(60.000:120.000:1.500000 and 4.765)arc(120.000:180.000:1.500000 and
    4.765);

  \path[shift={(0,-237.0)},draw=cc837ab,line cap=round,miter limit=4.00,nonzero
    rule,line width=0.2mm] (30.0000,247.5114)arc(-90.000:0.000:7.000002 and
    19.243)arc(0.000:90.000:7.000002 and 19.243);

\end{scope}
\begin{scope}[blend mode=screen,opacity=0.400,transparency group]% layer5
  \path[rotate=90.0,fill=cd8d4cc,line cap=round,miter limit=4.00,nonzero rule]
    (11.5027,-29.1833) .. controls (11.6755,-19.0910) and
    (19.9062,-11.0000) .. (29.9996,-11.0000) .. controls (27.9289,-11.0000) and
    (26.1698,-11.3619) .. (25.8617,-11.8512) .. controls (25.5535,-12.3405) and
    (26.7890,-12.8101) .. (28.7762,-12.9612) .. controls (26.9731,-16.3688) and
    (25.8786,-23.5800) .. (26.0107,-31.2422) .. controls (26.1428,-38.9045) and
    (27.4756,-45.5113) .. (29.3776,-47.9876) .. controls (19.2788,-47.6436) and
    (11.3299,-39.2757) .. (11.5027,-29.1833) -- cycle;

\end{scope}
\begin{scope}% layer13
  \path[shift={(0,-237.0)},draw=cc837ab,line cap=round,miter limit=4.00,nonzero
    rule,line width=0.2mm] (30.0000,248.5036)arc(270.000:348.874:5.840333 and
    18.248);

  \path[rotate around={90.0:(118.5,-118.5)},draw=black,line cap=round,miter
    limit=4.00,nonzero rule,line width=0.2mm]
    (265.7847,-12.9667)arc(116.300:183.749:5.000000 and
    19.000)arc(183.749:251.198:5.000000 and 19.000);

  \path[shift={(0,-237.0)},draw=black,line cap=round,miter limit=4.00,nonzero
    rule,line width=0.2mm] (11.0000,266.9993)arc(180.000:261.441:1.000000 and
    4.184)arc(261.441:342.882:1.000000 and 4.184);

\end{scope}
\begin{scope}[blend mode=screen,opacity=0.400,transparency group]% layer11
  \path[fill=cd8d4cc,line cap=round,miter limit=4.00,nonzero rule]
    (29.1833,11.5027) .. controls (39.2757,11.3299) and
    (47.6436,19.2788) .. (47.9889,29.3678) .. controls (47.9098,27.2055) and
    (47.4459,25.6617) .. (46.9244,25.8270) .. controls (46.4030,25.9924) and
    (46.0000,27.8110) .. (46.0000,29.9998) .. controls (46.0000,31.7027) and
    (46.2859,33.2758) .. (46.7500,34.1270) .. controls (47.2141,34.9782) and
    (47.7859,34.9782) .. (48.2500,34.1270) .. controls (48.7141,33.2758) and
    (49.0000,31.7027) .. (49.0000,30.0001) .. controls (49.0000,36.7880) and
    (45.3786,43.0605) .. (39.5000,46.4545) .. controls (33.6214,49.8485) and
    (26.3786,49.8485) .. (20.5000,46.4545) .. controls (14.6214,43.0605) and
    (11.0000,36.7880) .. (11.0000,30.0000) .. controls (11.0000,19.9062) and
    (19.0910,11.6755) .. (29.1833,11.5027) -- cycle;

\end{scope}
\begin{scope}% layer16
  \path[draw=black,line cap=round,miter limit=4.00,nonzero rule,line
    width=0.2mm]
    (49.0000,30.0000)arc(-0.000:60.000:19.000)arc(60.000:120.000:19.000)arc(120.000:180.000:19.000);

  \path[draw=black,line cap=round,miter limit=4.00,nonzero rule,line
    width=0.2mm]
    (11.0000,30.0000)arc(180.000:269.019:18.500)arc(269.019:358.038:18.500);

  \path[shift={(0,-237.0)},draw=cc837ab,line cap=round,miter limit=4.00,nonzero
    rule,line width=0.2mm] (30.0000,285.9978)arc(90.000:150.000:6.000002 and
    18.747)arc(150.000:210.000:6.000002 and 18.747)arc(210.000:270.000:6.000002
    and 18.747);

\end{scope}
\begin{scope}[blend mode=screen,opacity=0.400,transparency group]% layer10
  \path[fill=cd8d4cc,line cap=round,miter limit=4.00,nonzero rule]
    (45.0000,30.0002) .. controls (45.0000,26.6784) and
    (45.8954,23.9856) .. (47.0000,23.9856) .. controls (48.1046,23.9856) and
    (49.0000,26.6784) .. (49.0000,30.0001) .. controls (49.0000,19.2304) and
    (40.2696,10.5000) .. (29.5000,10.5000) .. controls (18.7304,10.5000) and
    (10.0000,19.2304) .. (10.0000,30.0000) .. controls (10.0000,37.1453) and
    (13.8120,43.7479) .. (20.0000,47.3205) .. controls (26.1880,50.8932) and
    (33.8120,50.8932) .. (40.0000,47.3205) .. controls (46.1880,43.7479) and
    (50.0000,37.1453) .. (50.0000,30.0000) .. controls (50.0000,32.2606) and
    (49.5235,34.3494) .. (48.7500,35.4797) .. controls (47.9765,36.6100) and
    (47.0235,36.6100) .. (46.2500,35.4797) .. controls (45.4765,34.3494) and
    (45.0000,32.2606) .. (45.0000,30.0000);

\end{scope}
\begin{scope}% layer17
  \path[shift={(0,-237.0)},draw=black,line cap=round,miter limit=4.00,nonzero
    rule,line width=0.2mm]
    (10.0000,267.0000)arc(180.000:270.000:19.500)arc(-90.000:0.000:19.500);

  \path[shift={(0,-237.0)},draw=black,line cap=round,miter limit=4.00,nonzero
    rule,line width=0.2mm]
    (50.0000,267.0000)arc(-0.000:60.000:20.000)arc(60.000:120.000:20.000)arc(120.000:180.000:20.000);

  \path[shift={(0,-237.0)},draw=cc837ab,line cap=round,miter limit=4.00,nonzero
    rule,line width=0.2mm] (30.0000,287.0000)arc(90.000:150.000:7.000002 and
    19.744)arc(150.000:210.000:7.000002 and 19.744)arc(210.000:270.000:7.000002
    and 19.744);

\end{scope}
\begin{scope}[blend mode=screen,opacity=0.400,transparency group]% layer6
  \path[fill=cd8d4cc,line cap=round,miter limit=4.00,nonzero rule]
    (29.5000,9.5000) .. controls (40.8218,9.5000) and
    (50.0000,18.6782) .. (50.0000,30.0000) .. controls (50.0000,26.6775) and
    (49.1465,23.7409) .. (47.8934,22.7516) .. controls (46.6402,21.7622) and
    (45.2784,22.9498) .. (44.5321,25.6830) .. controls (43.7857,28.4163) and
    (43.8282,32.0602) .. (44.6409,34.6819) .. controls (39.7060,36.1380) and
    (32.2272,36.7979) .. (24.9614,36.4232) .. controls (17.6956,36.0485) and
    (11.7165,34.6945) .. (9.2222,32.8555) .. controls (9.0746,31.9455) and
    (9.0000,30.9773) .. (9.0000,30.0000) .. controls (9.0000,18.6782) and
    (18.1782,9.5000) .. (29.5000,9.5000) -- cycle;

\end{scope}
\begin{scope}[shift={(0,-237.0)}]% layer1
  \path[draw=cc837ab,line cap=round,miter limit=4.00,nonzero rule,line
    width=0.2mm] (22.4234,273.2520)arc(161.276:215.638:8.000002 and
    20.248)arc(215.638:270.000:8.000002 and 20.248);

  \path[draw=black,line cap=round,miter limit=4.00,nonzero rule,line
    width=0.2mm] (50.0000,267.0000)arc(0.000:60.000:2.500003 and
    6.327)arc(60.000:120.000:2.500003 and 6.327)arc(120.000:180.000:2.500003 and
    6.327);

  \path[draw=black,line cap=round,miter limit=4.00,nonzero rule,line
    width=0.2mm] (45.0000,267.0002)arc(180.000:270.000:1.999999 and
    6.015)arc(270.000:360.000:1.999999 and 6.015);

  \path[draw=black,line cap=round,miter limit=4.00,nonzero rule,line
    width=0.2mm] (49.0000,267.0002)arc(-0.000:60.000:1.500000 and
    4.765)arc(60.000:120.000:1.500000 and 4.765)arc(120.000:180.000:1.500000 and
    4.765);

  \path[draw=black,line cap=round,miter limit=4.00,nonzero rule,line
    width=0.2mm] (46.0000,266.9993)arc(180.000:265.667:1.000000 and
    4.184)arc(265.667:351.333:1.000000 and 4.184);

  \path[draw=black,line cap=round,miter limit=4.00,nonzero rule,line
    width=0.2mm]
    (9.0000,267.0000)arc(180.000:270.000:20.500)arc(-90.000:0.000:20.500);

  \path[draw=black,line cap=round,miter limit=4.00,nonzero rule,line
    width=0.2mm] (44.6368,271.6774)arc(141.974:214.649:3.000002 and
    7.593)arc(214.649:287.325:3.000002 and 7.593)arc(287.325:360.000:3.000002 and
    7.593);

  \path[rotate=90.0,draw=black,line cap=round,miter limit=4.00,nonzero rule,line
    width=0.2mm] (271.6864,-44.6449)arc(-48.159:11.088:5.526316 and
    21.000)arc(11.088:70.335:5.526316 and 21.000);

  \path[draw=black,line cap=round,miter limit=4.00,nonzero rule,line
    width=0.2mm] (9.2195,269.8512)arc(157.944:180.000:3.000002 and 7.593);

\end{scope}

\end{scope}
}

\theoremstyle{definition}

\theoremstyle{plain}
\newtheorem{prop}{Proposition}
\newtheorem{lem}{Lemma}
\newtheorem{thm}{Theorem}
\newtheorem{cor}{Corollary}

\usetikzlibrary{cd, calc}

\definecolor{smoked}{RGB}{216, 212, 204}
\definecolor{mauve}{RGB}{200, 55, 171}

\newcommand{\CP}[1]{{\mathbb{CP}^{#1}}}

\pdfpageattr{/Group <</S /Transparency /I true /CS /DeviceRGB>>}

\usepackage{hyperref}
\hypersetup{
	colorlinks,
	linkcolor={mauve},
	citecolor={mauve},
	urlcolor={mauve}
}

\title{The complex geometry of the free particle, and its perturbations}
\author{Aaron Fenyes}
\date{}

\begin{document}
\maketitle
\begin{abstract}
The Hamiltonian operator describing a quantum particle on a path often extends holomorphically to a complex neighborhood of the path. When it does, it can be seen as the local expression of a complex projective structure, and its perturbations become deformations of that geometric structure.

We'll describe the Hamiltonian of a free particle as a complex projective surface, and we'll use tools from quasiconformal geometry to study its perturbations. Our main results are loosely modeled on the algebraic ``transformation theory'' results that underpin the exact WKB method. They're meant to serve as a foundation for efforts to gain a more geometric understanding of the exact WKB method.
\end{abstract}
\section{Introduction}
\subsection{Quantum mechanics on a complex domain}\label{context}
The motion of a quantum particle on a one-dimensional path is described by the Hamiltonian operator
\[ \left(\frac{\partial}{\partial z}\right)^2 - \frac{q}{2}, \]
where $z$ is the position variable and $q/2$ is the function that gives the potential energy of the particle at each point on the path. This kind of operator is called a Hill's operator~\cite{geom-kdv}.

Sometimes, a particle on a path has a chance of escaping the path altogether. For example, a photon in an optical fiber might escape through the side of the fiber, or a neutron passing through an atom's nucleus might be absorbed into the nucleus. One way of modeling this, widely used in nuclear physics~\cite{nuclear-optical}, is to make the potential energy function complex-valued, with non-positive imaginary part. The imaginary part of the potential gives the particle's chance of escape at each point on the path.

In nuclear physics, the complex potential often extends to a holomorphic function on a neighborhood $\Omega \subset \C$ of the real interval parameterizing the path~\cite[\S 3.3]{semiclass-scat}. When it does, we can treat $z$ as a complex variable and $\tfrac{\partial}{\partial z}$ as a complex derivative, reinterpreting the Hamiltonian as an operator on holomorphic functions over $\Omega$. Each holomorphic eigenfunction, if it's tame enough,\footnote{The details depend on which rigged Hilbert space we use to define unbound states~\cite[\S 8.4]{capri-nr-qm}\cite[\S 1.4, end of \S 10.1]{qm-modern}. In good cases, the tame holomorphic eigenfunctions correspond one-to-one with the generalized eigenfunctions we find in the rigged Hilbert space. See Appendix~\ref{gen-eigfns} for some examples.} describes a stationary or steadily decaying state of the particle---a bound state if it restricts to an $L^2$ function on the path, and an unbound state\footnote{Or, equivalently, a scattering state: for stationary and steadily decaying states, typical notions of unbound and scattering states coincide. You can find this point of view in the table at the end of Chapter~IV of \cite{dirac-gamow-gelfand}, or the discussion after equation~(1.2) of \cite{rhs-gamow}. Some typical notions of scattering states are described in \cite[\S 2.1]{phys-math-gamow}, \cite[\S 2]{leaky-modes}, and \cite[\S\S 16.3, 16.5]{qm-modern}.} otherwise. We can therefore study the motion of the particle using techniques from complex analysis and geometry.

One fruitful approach, called the {\em exact WKB method}, is rooted in the asymptotic analysis of solutions of complex differential equations~\cite{semiclass-scat}\cite{quartic-return}\cite{alg-anal-sing}\cite{exact-wkb-cluster}. The aim of this article is to illustrate another approach, based on the geometry of complex projective surfaces. Our main results are loosely modeled on the exact WKB ``transformation theory'' results of Kamimoto, Koike, Aoki, Kawai, and Takei \cite{borel-tfm-ser}. They're meant to serve as a foundation for efforts to gain a more geometric understanding of the exact WKB method.

Results similar to ours could be achieved in a more analytic way using Olver's error bounds for the Liouville-Green approximation~\cite[\S 6]{asymp-special}. Both approaches have their own advantages, and I think they could be usefully combined. I'll say more after stating the results.
\subsection{Conformal maps and the WKB method}
The first step of the WKB method is to transform the Hill's operator under consideration into something resembling the Hamiltonian of a free particle. The transformation that accomplishes this is an example of a {\em Liouville transformation} \cite[\S 6.1]{asymp-special}. On a complex domain, Liouville transformations pick up an especially geometric flavor, motivating the geometric point of view this article takes.

Say we have a Hill's operator $\mathcal{H}$ defined on an open region $\Omega \subset \C$, and another Hill's operator
\[ \tilde{\mathcal{H}} = \left(\frac{\partial}{\partial z}\right)^2 - \frac{\tilde{q}}{2} \]
defined on another open region $\tilde{\Omega} \subset \C$. A {\em Liouville transformation} that relates $\mathcal{H}$ to $\tilde{\mathcal{H}}$ is a conformal map $y \maps \Omega \to \tilde{\Omega}$ with the property that
\[ \mathcal{H} = y_z^{3/2} \circ \left[ \left(\frac{\partial}{\partial y}\right)^2 - \frac{y^* \tilde{q}}{2} \right] \circ y_z^{1/2}. \]

Those odd-looking powers of $y_z$ give Liouville transformations many nice features, discussed in Sections \ref{geom-cp} and \ref{deform-cp}. Their usefulness is perplexing if you treat $\mathcal{H}$ and $\tilde{\mathcal{H}}$ as operators on holomorphic functions over $\C$, as we'll do throughout this article. You can make geometric sense of them, however, by letting $\mathcal{H}$ and $\tilde{\mathcal{H}}$ act on holomorphic sections of the anti-tautological bundle $\mathcal{O}(1)$ over $\CP{1}$, sending them to holomorphic sections of $\mathcal{O}(-3)$.
\subsection{A free enough particle on a complex strip}
A basic system in quantum mechanics is the free particle: a particle that moves along the real line with constant potential energy, and no chance of escape. Its Hamiltonian is
\[ \left(\frac{d}{dz}\right)^2 - \frac{1}{4} \]
when we set the potential to $\tfrac{1}{4}$.

Let's consider a particle that's merely ``free enough'': its complex potential can vary with position, but not by too much. We'll express its Hamiltonian as
\[ \left(\frac{d}{dz}\right)^2 - \left(\frac{1}{4} + \frac{p}{2}\right), \]
where $p$ is a ``small enough'' function of position. Let's say $p$ is holomorphic on a horizontal strip $B \subset \C$ of height at least $\pi$.

In order to state our results cleanly, it will be useful to introduce a ``doubly extended'' complex plane $\hat{\C}$ with two extra points, $-\infty$ and $+\infty$. The left half-plane is a neighborhood of $-\infty$, and the right half-plane is a neighborhood of $+\infty$. Let $\hat{B}$ be the closure of $B$ in $\hat{\C}$. It comprises the strip $B$, the edges of the strip, and the points $-\infty$ and $+\infty$. If $B$ has finite height, $\hat{B}$ is topologically a closed disk.

Now we can state our two main results. We'll present them here in analytic language, with the Hamiltonian operators appearing explicitly. Before we prove these results, we'll recast them in geometric language as Theorems \ref{bicorn-qual-perturb} and \ref{bicorn-quant-perturb}, which can be found in Section~\ref{main-results}. The geometric statements are sleeker, for two reasons. First, by viewing the Hamiltonian operators as geometric structures, we avoid the need to write them out explicitly, as discussed in Section~\ref{oper-as-geom}. Second, the strange-looking function bounding the size of $p$ will turn out to be a coordinate description of a natural Riemannian metric on $B$, which we'll meet in Section~\ref{thurston-metric}.
\begin{thm}\label{elem-qual-perturb}
Consider a Hill's operator
\[ \mathcal{H} = \left(\frac{d}{dz}\right)^2 - \left(\frac{1}{4} + \frac{p}{2}\right) \]
on a horizontal strip $B \subset \C$ of height at least $\pi$, with $p$ holomorphic. Suppose $p$ is small in the sense that
\[ \left|\frac{p}{2}\right| < \begin{cases}
\left(\frac{1}{2 \sin \ell}\right)^2 & \ell \le \frac{\pi}{2} \\
\frac{1}{4} & \ell \ge \frac{\pi}{2},
\end{cases} \]
where $\ell$ is the function that gives the distance to the edge of $B$. Then we can find a new coordinate $y \maps B \to \C$ with the following properties.
\begin{itemize}
\item It extends to extends to a topological embedding $\hat{B} \to \hat{\C}$ which fixes $\pm\infty$.
\item It's a Liouville transformation that relates $\mathcal{H}$ to the Hamiltonian of a free particle with potential $\tfrac{1}{4}$. In other words,
\[ \mathcal{H} = y_z^{3/2} \circ \left[ \left(\frac{d}{dy}\right)^2 - \frac{1}{4} \right] \circ y_z^{1/2}. \]
\end{itemize}
\end{thm}
The strength of this theorem is that the new coordinate $y$ extends continuously to $\hat{B}$, even though $p$ can blow up at the boundary of $B$. If we improve the bound on $p$, we can get a more quantitative version of this result.
\begin{thm}\label{elem-quant-perturb}
Assume the setting of Theorem~\ref{elem-qual-perturb}. Suppose $p$ satisfies the stronger bound
\[ \left|\frac{p}{2}\right| < \begin{cases}
M\left(\frac{1}{2 \sin \ell}\right)^2 & \ell \le \frac{\pi}{2} \\
M\frac{1}{4} & \ell \ge \frac{\pi}{2}
\end{cases} \]
for some $M \in (0, \tfrac{1}{3})$. Then any coordinate $y$ with the properties offered by Theorem~\ref{elem-qual-perturb} is close to a translation. Specifically, for any two points $\mathfrak{m}, \mathfrak{n} \in \hat{B} \smallsetminus \{\pm\infty\}$, the displacements $r = z(\mathfrak{n}) - z(\mathfrak{m})$ and $r' = y(\mathfrak{n}) - y(\mathfrak{m})$ are kept close to each other by the inequality
\[ |r' - r| \le \tfrac{2M}{1-M}(|\Re r| + \tfrac{3}{2}j\pi), \]
where $j \in \{1, 2, 3, \ldots\}$ and $|\Im r| \in [0, j\pi]$.
\end{thm}
For the applications to the exact WKB method that I have in mind, $p$ will have nothing worse than isolated power-law singularities on the boundary of $B$. In this case, we could get similar results from Olver's error bounds for the Liouville-Green approximation~\cite[\S 6.11]{asymp-special}, using the tricks in \cite[\S 6.4]{asymp-special} to show that the error control function has bounded variation along a progressive path to each singularity. The bound on $|r' - r|$ might be more complicated near the singularities, especially regular singularities. However, it might also be tighter, especially where $p$ is small. I expect that Theorem~\ref{elem-quant-perturb} could be improved by using our geometric approach near the boundary of $B$, in the region $\ell \le \tfrac{\pi}{2}$, and switching to Olver's approach in the region $\ell \ge \tfrac{\pi}{2}$.

Here are some instructions for translating the analytic statements above into the geometric statements of Section~\ref{main-results}. They won't be immediately usable, because they rely on the geometric language we'll introduce on the way from here to Section~\ref{free-enuf-geom}. I've placed them here as a guide to how that geometric language is relevant to the analytic problem we originally set out to solve.
\begin{proof}[Reduction of Theorem~\ref{elem-qual-perturb} to Theorem~\ref{bicorn-qual-perturb}]
An open region in $\C$ equipped with a Hill's operator can be seen as a chart for a complex projective surface (Section~\ref{oper-as-geom}). The complex plane, equipped with our free particle Hamiltonian, is the image of the global chart $u \maps U \to \C$ for the infinite bicorn $U$ (Section~\ref{free-geom}). The strip $B$, equipped with the same Hamiltonian, is a global chart for a bicorn (Section~\ref{bicorn-defn}).

Let $\varepsilon$ be the quadratic differential $-p\,dz^2$ on $B$. Using a coordinate expression for the Thurston metric of $B$ (Section~\ref{thurston-metric}), you can see that the bound on $p$ in Theorem~\ref{elem-qual-perturb} is equivalent to the bound on $\varepsilon$ in Theorem~\ref{bicorn-qual-perturb}. The conformal embedding $\gamma \maps B \to U$ that we get from Theorem~\ref{bicorn-qual-perturb} appears in coordinates, by composition with $u$, as a conformal map $y \maps B \to \C$.

The space $\hat{\C}$ is a coordinate description of $\hat{U}$, so $\gamma$ extends to a topological embedding $\hat{B} \to \hat{U}$ if and only if $y$ extends to a topological embedding $\hat{B} \to \hat{\C}$. Using the properties of Schwarzian derivatives outlined at the end of Section~\ref{deform-cp}, you can work out that $\gamma$ has Schwarzian derivative $\varepsilon$ if and only if $y$ is a Liouville transformation that relates $\mathcal{H}$ to our free particle Hamiltonian.
\end{proof}
\begin{proof}[Reduction of Theorem~\ref{elem-quant-perturb} to Theorem~\ref{bicorn-quant-perturb}]
We give the infinite bicorn a translation structure by declaring $u$ to be a translation chart (Section~\ref{tras-struct}). Expressed in terms of this chart, the conclusion of Theorem~\ref{bicorn-quant-perturb} becomes the conclusion of Theorem~\ref{elem-quant-perturb}.
\end{proof}
\subsection{Acknowledgements}
This line of work was suggested by Marco Gualtieri, informed by Shinji Sasaki, supported by the University of Toronto, and pursued in part at the Iowa City Public Library. I'm grateful to Francis Bischoff, Eduard Duryev, Marco Gualtieri, Omar Kidwai, Tatsuya Koike, Leonid Monin, Nikita Nikolaev, Shinji Sasaki, and Israel Michael Sigal for their feedback on the results and presentation of this paper.
\section{Hill's operators as geometric structures}\label{oper-as-geom}
\subsection{Overview}
The analysis of holomorphic functions on open regions in $\C$ leads naturally to the geometric study of Riemann surfaces and conformal maps. The analysis of Hill's operators on open regions in $\C$ leads to its own kind of geometry---the study of {\em complex projective surfaces} and {\em M\"{o}bius maps}~\cite{geom-kdv}.

When you take a geometric point of view, the problem of finding a Liouville transformation that relates one Hill's operator to another becomes the problem of finding a M\"{o}bius map from one complex projective surface to another. This is useful because complex projective surfaces, like Riemann surfaces, all look the same on small scales; the only obstacles to finding M\"{o}bius maps between them are large-scale mismatches. Complex projective surfaces are also very rigid: M\"{o}bius maps are determined by their local behavior, and their local behavior has few degrees of freedom. These properties aren't obvious from an analytic point of view. We'll see how they arise in Section~\ref{geom-cp}.

The perturbation theory of Hill's operators appears geometrically as a rich deformation theory of complex projective structures. Its basic ideas are outlined in Section~\ref{deform-cp}. The WKB method is a perturbative analysis, and the main results of this paper are, correspondingly, deformation theory results.
\subsection{Analytic definition}
A conformal structure on a surface is an atlas of charts to $\C$ whose transition maps are conformal. A {\em complex projective structure} is a conformal structure that comes with a Hill's operator on the image of each chart. Its transition maps are required to be Liouville transformations.

A {\em M\"{o}bius map} is a structure-preserving map between complex projective surfaces: a conformal map which becomes a Liouville transformation when composed with charts on both sides.
\subsection{Geometric simplification}\label{geom-cp}
You might expect small pieces of complex projective surfaces to come in many different shapes, arising from Hill's operators with different sorts of potentials. Surprisingly, they don't! All complex projective surfaces are locally isomorphic. In other words, all Hill's operators are locally Liouville-equivalent. The construction behind this observation dates back to Kummer, although the observation may not~\cite[\S 2]{old-and-new}. It's reviewed in Appendix~\ref{kummer-tfm}.

Using Kummer's trick, you can describe any complex projective structure in terms of {\em projective charts}, which come with the operator $(\frac{\partial}{\partial z})^2$ on their images. The Liouville transformations relating this operator to itself are the M\"{o}bius transformations. If you stick to projective charts, you can forget about Hill's operators entirely, and say that a complex projective structure is an atlas of charts to $\CP{1}$ whose transition maps are M\"{o}bius transformations. A M\"{o}bius map is a conformal map which becomes a M\"{o}bius transformation when composed with projective charts on both sides.
\subsection{Deformations}\label{deform-cp}
If you have a conformal map $\zeta$ from one complex projective surface to another, you can see how it deforms the complex projective structure by measuring its {\em Schwarzian derivative} $\mathcal{S}\zeta$. This holomorphic quadratic differential tracks how $\zeta$ deviates from its best M\"{o}bius approximation at each point. Its description here follows the more detailed treatment in \cite[\S 4.2]{warping-geom}, which in turn is based on an exposition by Thurston~\cite{zippers}.

Pick a point $\mathfrak{m}$ in the domain of $\zeta$. To reduce clutter---hopefully without adding too much confusion---we'll also use $\mathfrak{m}$ to denote the ideal of holomorphic functions vanishing at $\mathfrak{m}$. For any projective chart $v$ vanishing at $\mathfrak{m}$, you can find a projective chart $\tilde{v}$ around $\zeta(\mathfrak{m})$ with
\[ \zeta^* \tilde{v} \in v + \mathfrak{m}^3. \]
Intuitively, this shows that $\zeta$ is a M\"{o}bius map through second order near $\mathfrak{m}$.

To see how far $\zeta$ is from actually being a M\"{o}bius map, let's look at its third-order expansion~\cite[\S 4.2.2]{warping-geom}. Rewriting the second-order expansion above as
\[ \zeta^* \tilde{v} \in v(1 + \mathfrak{m}^2), \]
and expanding to the next order inside the parentheses, we get
\[ \zeta^* \tilde{v} \in v(1 + \tfrac{1}{6} p_\mathfrak{m}\,v^2 + \mathfrak{m}^3) \]
for some $p_\mathfrak{m} \in \C$. More abstractly,\footnote{The factor of $\tfrac{1}{6}$ gives the standard normalization for the Schwarzian derivative. The normalization used in \cite[\S 4.2.2]{warping-geom} is not standard.}
\[ \zeta^* \tilde{v} \in v(1 + \tfrac{1}{6}\,\mathcal{S}\zeta_\mathfrak{m}) \]
for some coset $\mathcal{S}\zeta_\mathfrak{m} \in \mathfrak{m}^2/\mathfrak{m}^3$, which can be written in coordinates as $p_\mathfrak{m}\,v^2 + \mathfrak{m}^3$. This coset, which turns out to be the same for any choice of $v$~\cite[\S 4.2.3]{warping-geom}, is the Schwarzian derivative of $\zeta$ at $\mathfrak{m}$.

You can identify the coset $v^2 + \mathfrak{m}^3$ with the square of the cotangent vector $dv_\mathfrak{m}$~\cite[\S 4.2.2]{warping-geom}. This leads to the expression $\mathcal{S}\zeta_\mathfrak{m} = p_\mathfrak{m}\,dv_\mathfrak{m}^2$, and the interpretation of $\mathcal{S}\zeta$ as a quadratic differential. If you fix the chart $v$, the approximating chart $\tilde{v}$ depends holomorphically on $\mathfrak{m}$, so $\mathcal{S}\zeta$ is holomorphic. In coordinates, $\mathcal{S}\zeta = p\,dv^2$ for some holomorphic function $p$.

The Schwarzian derivative provides a useful perspective on Kummer's construction of local projective charts. To see it, you need to know two things. One is that when you compose two conformal maps between complex projective surfaces, their Schwarzian derivatives add: $\mathcal{S}(\omega \circ \zeta) = \mathcal{S}\zeta + \zeta^* \mathcal{S}\omega$~\cite[p. 188]{zippers}\cite[\S 3.1]{cp-structs}. The other is that you can find the Schwarzian derivative of a chart for a complex projective structure by looking at the Hill's operator on its image. A chart $f$ that comes with the operator
\[ \mathcal{H} = \left(\frac{\partial}{\partial z}\right)^2 - \frac{q}{2} \]
on its image has Schwarzian derivative $f^*q\,df^2$.

Putting these facts together, we see that a composition $y \circ f$ is a M\"{o}bius map if and only if $\mathcal{S}y = -q\,dz^2$. Our discussion of Kummer's trick in Appendix~\ref{kummer-tfm} can therefore be reframed as a proof that if $\psi$ and $\tilde{\psi}$ satisfy the Hill's equations $\mathcal{H}\psi = 0$ and $\mathcal{H}\tilde{\psi} = 0$, their ratio $\tilde{\psi}/\psi$ has Schwarzian derivative $-q\,dz^2$. This is what Kummer originally observed~\cite[\S 2]{old-and-new}.
\section{The geometry of the free particle}\label{free-geom}
\subsection{Overview}
Let's look at a free particle from the geometric viewpoint of Section~\ref{oper-as-geom}. Equipped with the Hamiltonian
\[ \left(\frac{d}{dz}\right)^2 - \frac{1}{4}, \]
the complex plane becomes a global chart for some complex projective surface $U$. We'll see in this section that $U$ can be described very simply in geometric terms. We'll also parlay the translation symmetry of the free particle into extra geometric structure on $U$, which complements the complex projective structure.
\subsection{Geometric definition}
The twice-punctured Riemann sphere $\CP{1} \smallsetminus \{0, \infty\}$, also known as $\C^\times$, gets a natural complex projective structure from its inclusion into $\CP{1}$. Its universal cover $U \to \C^\times$ inherits that complex projective structure in the usual way. I'll call $U$ the {\em infinite bicorn}, because the covering map rolls it up like the brim of a bicorn hat.
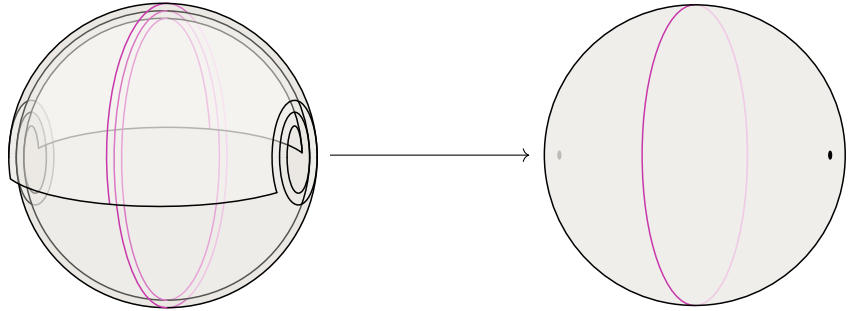
\begin{center}
\begin{tikzpicture}[nodes={align=center}]
\matrix[column sep=3cm, nodes={outer sep=2mm}]{
\inftbicorn{1}{bicorn}
&
\punkdsphere{1}{sphere}
\\
};
\path[
  commutative diagrams/.cd, every arrow, every label,
  nodes={fill=white, inner sep=1mm, outer sep=0.5mm}
]
  (bicorn) edge (sphere);
\end{tikzpicture}
\end{center}
For convenience, mark a point $\star \in U$ which lies above $1 \in \CP{1}$.
\subsection{Riemann surface description}
As a covering map between Riemann surfaces, $U \to \C^\times$ must be equivalent to the exponential map $\C \to \C^\times$. That means there's a unique conformal equivalence $u \maps U \to \C$ that sends $\star$ to $0$ and makes the diagram
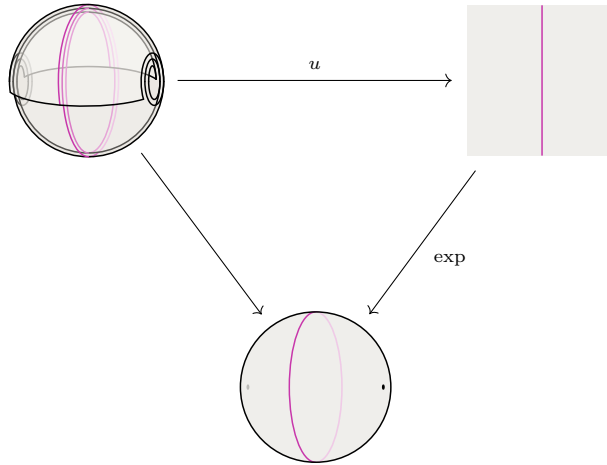
\begin{center}
\begin{tikzpicture}[nodes={align=center}]
\matrix[row sep=2cm, column sep=1cm, nodes={outer sep=2mm}]{
\inftbicorn{0.5}{bicorn} & &
\node[rectangle, fill=smoked!40, minimum size=0.5*4cm] (0, 0) (plane) {};
\draw[mauve, line width=0.2mm] (0, -1) -- (0, 1);
\\
&
\punkdsphere{0.5}{sphere}
\\
};
\path[
  commutative diagrams/.cd, every arrow, every label,
  nodes={fill=white, inner sep=1mm, outer sep=0.5mm}
]
  (bicorn) edge node {$u$} (plane)
  (bicorn) edge (sphere)
  (plane) edge node {$\exp$} (sphere);
\end{tikzpicture}
\end{center}
commute. We'll think of $u$ as a global chart for $U$, using the end of Section~\ref{deform-cp} to find out which Hill's operator it comes with. Working backwards from the fact that $\exp u$ is a M\"{o}bius map, you can infer that $u$ carries the free particle Hamiltonian
\[ \left(\frac{d}{dz}\right)^2 - \frac{1}{4} \]
on its image, and thus has Schwarzian derivative $\tfrac{1}{2} du^2$.
\subsection{Boundary points}
We can extend the projective cover $U \to \C^\times$ to a topological branched cover $\hat{U} \to \CP{1}$ by adding a pair of points $-\infty$ and $+\infty$ that map to $0$ and $\infty$, respectively. I'll call these new points the {\em tips} of the infinite bicorn. Give $\hat{U}$ the coarsest topology that makes the map $\hat{U} \to \CP{1}$ continuous.
\subsection{Translation structure}\label{tras-struct}
We saw earlier that $U$ is conformally equivalent to the flat plane, and we know that all conformal automorphisms of the flat plane are composed of translations, rotations, and scalings. Among these, translations and half-turn rotations are the only ones that preserve the Schwarzian derivative $\tfrac{1}{2} du^2$. The projective automorphisms of $U$ are therefore composed of translations and half-turn rotations---the local symmetries of a {\em half-translation surface}~\cite[\S 2]{billiards-problems}\cite[\S 1.8]{rat-flat}.

It's now clear that all the projective automorphisms of $U$ extend to topological automorphisms of $\hat{U}$, which either fix or exchange the tips. The ones that fix the tips are the pure translations.

In light of all this, let's define a {\em translation structure} on $U$ by declaring $u$ to be a translation chart~\cite[\S 2]{billiards-problems}. The global translations of $U$ are the projective automorphisms that fix the tips. Translation of $\hat{U}$ by $r \in \C$ commutes with multiplication of $\CP{1}$ by $e^r \in \C^\times$. In particular, vertical translations of $\hat{U}$ correspond to rotations of $\CP{1}$, so it's natural to think of vertical distances on $U$ as angular distances.

On a small region $\Omega \subset U$, the complex projective structure is more flexible than the translation structure: there are M\"{o}bius maps $\Omega \to U$ which are not translation maps. On a region that connects the tips, we can stiffen the complex projective structure by keeping the tips fixed.
\begin{prop}\label{bicorn-region-rigid}
Consider a connected open subset $\Omega \subset U$ and a M\"{o}bius map $\omega \maps \Omega \to U$. If the tips of $U$ are in the closure of $\Omega$, and $\omega$ extends to a continuous map $\Omega \cup \{\pm\infty\} \to \hat{U}$ that fixes the tips, then $\omega$ is a translation map.
\end{prop}
\begin{proof}
Any two M\"{o}bius maps from a connected complex projective surface to $\CP{1}$ differ by a M\"{o}bius transformation, so there's a M\"{o}bius transformation $\theta$ that makes the diagram
\begin{center}
\begin{tikzcd}
\Omega \arrow[d] \arrow[r, "\omega"] & U \arrow[d] \\
\CP{1} \arrow[r, "\theta"'] & \CP{1}
\end{tikzcd}
\end{center}
commute. Since $\omega$ extends continuously to the tips of $U$, and fixes each tip, $\theta$ must fix $0$ and $\infty$. That means there's a translation $\omega'$ of $U$ that makes the diagram
\begin{center}
\begin{tikzcd}
U \arrow[d] \arrow[r, "\omega'"] & U \arrow[d] \\
\CP{1} \arrow[r, "\theta"'] & \CP{1}
\end{tikzcd}
\end{center}
commute. We can choose $\omega'$ to agree with $\omega$ at some point in $\Omega$. It follows, by the uniqueness of lifts along covering maps, that $\omega'$ agrees with $\omega$ everywhere in $\Omega$. Hence, $\omega$ is a translation map.
\end{proof}
\subsection{Horizontal strips}\label{bicorn-defn}
A horizontal strip in the infinite bicorn $U$ will be called a {\em bicorn}---a {\em finite} one if it's vertically bounded, and a {\em half-infinite} one if it extends infinitely in one vertical direction. A bicorn of height $2\pi$ or less maps injectively to $\CP{1}$ along the covering map.
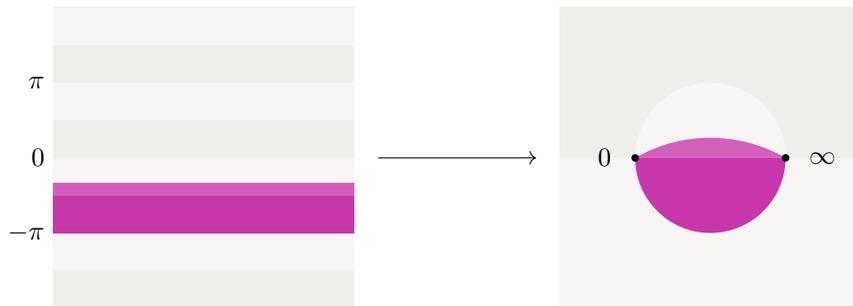
\begin{center}
\begin{tikzpicture}[nodes={align=center}]
\matrix[column sep=2.5cm, nodes={outer sep=2mm}]{
\node (bicorn) at (4, 0) {};
\bicornplane
\fill[mauve] (0, -1) rectangle (4, -1/3);
\fill[mauve!80] (0, -1/2) rectangle (4, -1/3);
&
\node (sphere) at (-2, 0) {};
\twopunkplane{mauve}{smoked!20}{smoked!40}{smoked!20}
\fill[mauve!80]
  (zero) -- (infty)
  arc[start angle=60, radius=2, delta angle=60];
\twopunks
\\
};
\path[
  commutative diagrams/.cd, every arrow, every label,
  nodes={fill=white, inner sep=1mm, outer sep=0.5mm}
]
  (bicorn) edge (sphere);
\end{tikzpicture}
\end{center}
The image of a taller bicorn overlaps itself.
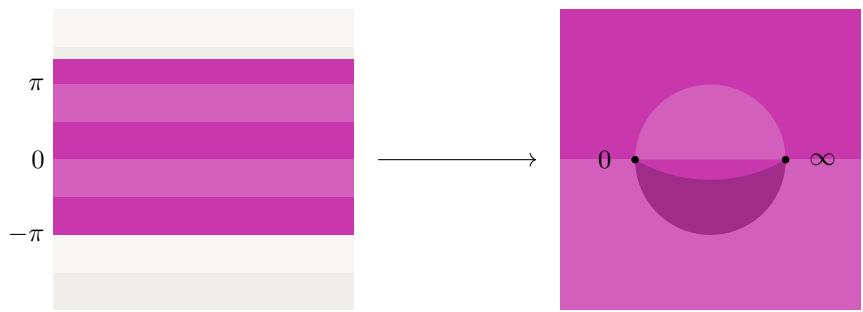
\begin{center}
\begin{tikzpicture}[nodes={align=center}]
\matrix[column sep=2.5cm, nodes={outer sep=2mm}]{
\node (bicorn) at (4, 0) {};
\bicornplane
\fill[mauve] (0, -1) rectangle (4, 1+1/3);
\fill[mauve!80] (0, -1/2) rectangle (4, 0);
\fill[mauve!80] (0, 1/2) rectangle (4, 1);
&
\node (sphere) at (-2, 0) {};
\twopunkplane{mauve}{mauve!80}{mauve}{mauve!80}
\fill[mauve!80!black]
  (zero)
  arc[start angle=180, radius=1, delta angle=180]
  arc[start angle=-60, radius=2, delta angle=-60];
\twopunks
\\
};
\path[
  commutative diagrams/.cd, every arrow, every label,
  nodes={fill=white, inner sep=1mm, outer sep=0.5mm}
]
  (bicorn) edge (sphere);
\end{tikzpicture}
\end{center}
A bicorn of height $\pi$ will be called a {\em bicorn disk}, because it maps to a disk in $\CP{1}$.
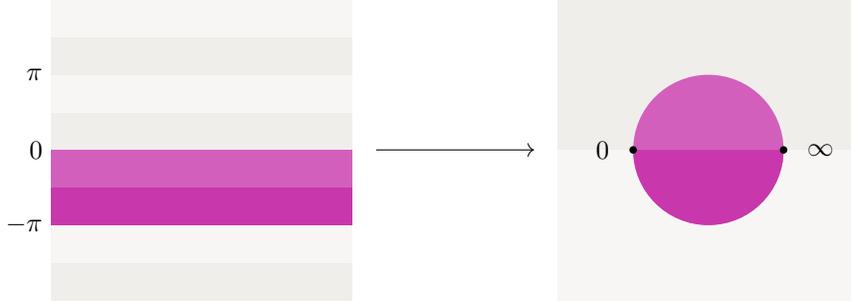
\begin{center}
\begin{tikzpicture}[nodes={align=center}]
\matrix[column sep=2.5cm, nodes={outer sep=2mm}]{
\node (bicorn) at (4, 0) {};
\bicornplane
\fill[mauve] (0, -1) rectangle (4, 0);
\fill[mauve!80] (0, -1/2) rectangle (4, 0);
&
\node (sphere) at (-2, 0) {};
\twopunkplane{mauve}{mauve!80}{smoked!40}{smoked!20}
\twopunks
\\
};
\path[
  commutative diagrams/.cd, every arrow, every label,
  nodes={fill=white, inner sep=1mm, outer sep=0.5mm}
]
  (bicorn) edge (sphere);
\end{tikzpicture}
\end{center}
Bicorn disks have many special features, which will play an important role in this paper.

%---
A bicorn comes with a complex projective structure, a translation structure, and a boundary---its topological boundary in $\hat{U}$. I'll call this boundary the {\em circular boundary} to avoid ambiguity. The closure in $\hat{U}$ of a bicorn $B$ will be called its {\em circular closure}, and denoted $\hat{B}$.

In addition to the tips $-\infty$ and $+\infty$, the circular boundary of a finite bicorn contains two horizontal lines in $U$, which I'll call the {\em edges}. The circular boundary of a half-infinite bicorn has one edge, and the circular boundary of the infinite bicorn has no edges.

You can describe a bicorn $B$ intrinsically by giving its complex projective structure, the topological details of its circular closure $B \hookrightarrow \hat{B}$, and the locations of its tips $-\infty$ and $+\infty$ in $\hat{B}$. This description, by Proposition~\ref{bicorn-region-rigid}, determines the inclusion $B \hookrightarrow U$ up to translation, which is to say it determines the translation structure of $B$. The intrinsic point of view lets us handle a bicorn as an independent object, not equipped with any particular inclusion into $U$.

We just observed that, up to translation, there's only one M\"{o}bius map $B \to U$ that sends each tip of $B$ to the corresponding tip of $U$. We get the same kind of rigidity for any conformal map $B \to U$ with a fixed Schwarzian derivative.

\begin{prop}\label{bicorn-deform-rigid}
Consider two conformal maps from a bicorn $B$ to the infinite bicorn $U$. Suppose they extend to continuous maps $B \cup \{\pm\infty\} \to U \cup \{\pm\infty\}$ that send each tip of $B$ to the corresponding tip of $U$. If the two maps have the same Schwarzian derivative, they differ only by a translation of $U$.
\end{prop}
\begin{proof}
Call the two maps $\zeta, \zeta' \maps B \to U$. Let $\omega \maps \zeta(B) \to U$ be the conformal map for which $\omega \circ \zeta = \zeta'$. Observe that $\omega$ extends continuously to the tips of $U$, fixing each tip.

Now, suppose $\zeta$ and $\zeta'$ have the same Schwarzian derivative. Then $\omega$ is a M\"{o}bius map. Proposition~\ref{bicorn-region-rigid} guarantees that $\omega$ is a translation map, which extends to a translation of $U$.
\end{proof}
\subsection{The Thurston metric}\label{thurston-metric}
Every complex projective surface of hyperbolic type comes with a Riemannian metric called the Thurston metric. It's defined pointwise as the infimum of the hyperbolic metrics of all projectively immersed disks~\cite[\S 4.3]{cp-structs}.\footnote{Following the convention of our sources, we'll use the hyperbolic metric of curvature $-4$.} When one disk sits inside another, the hyperbolic metric of the larger disk is smaller at every point. That means the Thurston metric is determined by the maximal disks---the disks with no larger disks containing them.\footnote{When I say one projectively immersed disk sits inside another, I mean there's a M\"{o}bius map from one to the other that commutes with the immersions.}

The Thurston metric of a disk is the same as the hyperbolic metric. On the bicorn disk $\Im u \in (0, \pi)$ in $U$, it can be written as
\[ \left[\frac{|du|}{2\sin(\Im u)}\right]^2. \]
The maximal disks of the infinite bicorn are precisely the bicorn disks. From this fact, and the expression above, you can deduce that the Thurston metric of the infinite bicorn is $\tfrac{1}{4}|du|^2$---a constant multiple of the flat metric that comes from the translation structure. The Thurston metric of a general bicorn is
\[ \begin{cases}
\left(\frac{|du|}{2 \sin \ell}\right)^2 & \ell \le \frac{\pi}{2} \\
\frac{1}{4}|du|^2 & \ell \ge \frac{\pi}{2},
\end{cases} \]
where $\ell$ is the function that gives the distance to the circular boundary.
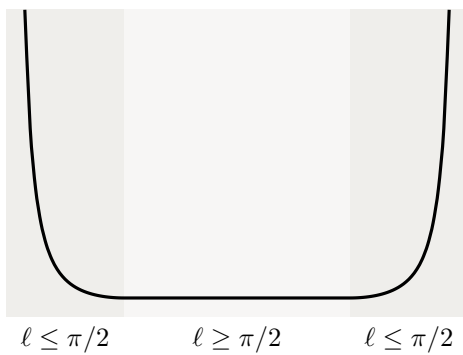
\newcommand{\graft}{3}
\pgfmathsetmacro{\lout}{0.25}
\pgfmathsetmacro{\lin}{pi/2}
\pgfmathsetmacro{\rin}{\graft+pi/2}
\pgfmathsetmacro{\rout}{\graft+pi-0.25}
\pgfmathsetmacro{\graphheight}{1/pow(2*sin(\lout r),2)}
\newcommand{\hang}{0.25}
\begin{center}
\begin{tikzpicture}
% axes
\fill[smoked!40] (0, 0) rectangle ({\graft+pi}, \graphheight);
\fill[smoked!20] (\lin, 0) rectangle (\rin, \graphheight);

% axis labels
\node[below] at (\lin/2, 0) {$\ell \le \pi/2$};
\node[below] at (\lin+\graft/2, 0) {$\ell \ge \pi/2$};
\node[below] at (\rin+\lin/2, 0) {$\ell \le \pi/2$};

% graph
\draw[very thick] plot[domain=\lout:\rout, samples=80, smooth] function{
  x < \lin ? 1/(2*sin(x))**2 : (x < \rin ? 1/4. : 1/(2*sin(x-\graft-pi))**2)
};
\end{tikzpicture}
\end{center}
The graph above shows how the size of the Thurston metric depends on your vertical position in a bicorn. Size is measured relative to $|du|^2$.
\section{The geometry of a free enough particle}\label{free-enuf-geom}
\subsection{Foundations}
Our main results are based on two classic extension theorems for local conformal maps from a complex projective surface to the Riemann sphere. The first, by Gehring and Pommerenke, extends a conformal map on a disk to a topological embedding on a larger domain. The second, by Ahlfors and Weill, bounds the maximal dilatation of the extension, making it quasiconformal.

We'll use these theorems through the following corollaries, chosen for their resemblance to the results we aim to achieve. The second theorem has been combined with Teichm\"{u}ller's observation that quasiconformal maps are close to M\"{o}bius maps, in the sense that they don't distort cross ratios too much.
\begin{cor}[Gehring--Pommerenke]\label{gp}
Consider a holomorphic quadratic differential $\varepsilon$ on a disk $D$. Write the hyperbolic metric on $D$ as $\eta^2$. If $|\varepsilon| < 2\eta^2$, there's a conformal embedding $D \to \CP{1}$ with the following properties.
\begin{itemize}
\item It extends to a topological embedding $\hat{D} \to \CP{1}$.
\item Its Schwarzian derivative is $\varepsilon$.
\end{itemize}
\end{cor}
\begin{cor}[Ahlfors--Weill, Teichm\"{u}ller]\label{awt}
If the quadratic differential $\varepsilon$ in Corollary~\ref{gp} satisfies the stronger bound $|\varepsilon| < 2M\eta^2$ for some $M \in (0, 1)$, the embedding $\hat{D} \to \CP{1}$ that we get is close to a M\"{o}bius map. Specifically, the cross ratio $\rho$ of four distinct points in $\hat{D}$ and the cross ratio $\rho'$ of their images in $\CP{1}$ are kept close to each other by the inequality
\[ d_{\therefore}(\rho, \rho') \le \tfrac{1}{2} \log \tfrac{1 + M}{1 - M}, \]
where $d_\therefore$ is the hyperbolic distance on $\CP{1} \smallsetminus \{0, 1, \infty\}$.
\end{cor}
\begin{proof}[Proof of Corollary~\ref{gp}]
Identify $\hat{D}$ with the closed unit disk in $\CP{1}$. Since $D$ is simply connected, there's a conformal map $D \to \CP{1}$ with Schwarzian derivative $\varepsilon$~\cite[Theorem~II.1.1]{unival-teich}. If $|\varepsilon| < 2\eta^2$ this map extends to a homeomorphism $\CP{1} \to \CP{1}$ by Gehring and Pommerenke's theorem~\cite[Theorem II.5.4]{unival-teich}. Restricting to $\hat{D}$ gives the desired topological embedding.
\end{proof}
\begin{proof}[Proof of Corollary~\ref{awt}]
As we did in the proof of Corollary~\ref{gp}, identify $\hat{D}$ with the closed unit disk in $\CP{1}$, and find a conformal map $D \to \CP{1}$ with Schwarzian derivative $\varepsilon$. If $|\varepsilon| < 2\eta^2$, this map extends to a quasiconformal homeomorphism $\mathring{\delta} \maps \CP{1} \to \CP{1}$ with a known Beltrami differential, found by Ahlfors and Weill~\cite[Theorem II.5.1]{unival-teich}. If $|\varepsilon| < 2M\eta^2$, the magnitude of the Beltrami differential is less than $M$ almost everywhere. It follows that $\mathring{\delta}$ is $\tfrac{1 + M}{1 - M}$-quasiconformal~\cite[Theorem I.4.1]{unival-teich}. Teichm\"{u}ller's theorem on the quasi-invariance of cross ratios then bounds $d_{\therefore}(\rho, \rho')$ as stated~\cite{teich-thm}.
\end{proof}
\subsection{Main results}\label{main-results}
\begin{thm}\label{bicorn-qual-perturb}
Consider a holomorphic quadratic differential $\varepsilon$ on a bicorn $B$ wide enough to contain a bicorn disk. Write the Thurston metric on $B$ as $\lambda^2$. If $|\varepsilon| < 2\lambda^2$, there's a conformal embedding $B \to U$ with the following properties.
\begin{itemize}
\item It extends to a topological embedding $\hat{B} \to \hat{U}$ that sends each tip of $B$ to the corresponding tip of $U$.
\item Its Schwarzian derivative is $\varepsilon$.
\end{itemize}
\end{thm}
\begin{thm}\label{bicorn-quant-perturb}
If the quadratic differential $\varepsilon$ in Theorem~\ref{bicorn-qual-perturb} satisfies the stronger bound $|\varepsilon| < 2M\lambda^2$ for some $M \in (0, \tfrac{1}{3})$, the embedding $\hat{B} \to \hat{U}$ that we get is close to a translation map. Specifically, the displacement $r$ between two points in $\hat{B} \smallsetminus \{\pm\infty\}$ and the displacement $r'$ between their images in $U$ are kept close to each other by the inequality
\[ |r' - r| \le \tfrac{2M}{1-M}(|\Re r| + \tfrac{3}{2}j\pi), \]
where $j \in \{1, 2, 3, \ldots\}$ and $|\Im r| \in [0, j\pi]$. Displacement is measured with respect to the standard translation structures on $B$ and $U$, as described in Section~\ref{tras-struct}.
\end{thm}
\subsection{Proofs}
\begin{proof}[Proof of Theorem~\ref{bicorn-qual-perturb}]
Since $B$ is wide enough to contain a bicorn disk, we can cover it with bicorn disks $D_1, \ldots, D_n$, ordered counterclockwise around $-\infty$. For convenience, let's arrange for the only overlaps to be between $D_\nu$ and $D_{\nu+1}$ for each $\nu$. Write the hyperbolic metric on $D_\nu$ as $\eta_\nu^2$.

Suppose $|\varepsilon| \le 2\lambda^2$. Since $D_\nu \subset B$ is a maximal disk, the definition of the Thurston metric ensures that $\lambda^2 \le \eta_\nu^2$. Corollary~\ref{gp} then gives us a conformal embedding $\delta \maps D_\nu \to \CP{1}$ with the following properties.
\begin{itemize}
\item Its Schwarzian derivative is $\varepsilon$.
\item It extends to a topological embedding $\hat{\delta} \maps \hat{D}_\nu \to \CP{1}$.
\end{itemize}
By composing with a suitable M\"{o}bius transformation, we can arrange for $\hat{\delta}_\nu$ to send $-\infty$ and $+\infty$ to $0$ and $\infty$, respectively. Then we can lift it along the branched covering $\hat{U} \to \CP{1}$ to a topological emedding $\hat{\gamma}_\nu \maps \hat{D}_\nu \to \hat{U}$.

By Proposition~\ref{bicorn-deform-rigid}, we can translate the embeddings $\hat{\gamma}_1, \ldots, \hat{\gamma}_n$ so that each one agrees with the previous one on the overlap between their domains. Then they fit together into one big topological embedding $\hat{B} \to \hat{U}$, which restricts to the desired conformal embedding $B \to U$.
\end{proof}
To prove Theorem~\ref{bicorn-quant-perturb}, we'll use a version of Corollary~\ref{awt} that not only tells us the embedding that we get is close to a M\"{o}bius map, but also shows us an efficient way to reach one.
\begin{lem}\label{hty-awt}
If the quadratic differential $\varepsilon$ in Corollary~\ref{gp} satisfies the stronger bound $|\varepsilon| < 2M\eta^2$ for some $M \in (0, \tfrac{1}{3})$, there's a homotopy $[0, 1] \times D \to \CP{1}$ with the following properties.
\begin{itemize}
\item It's conformal at each time.
\item It's a M\"{o}bius map at time zero.
\item It has Schwarzian derivative $\varepsilon$ at time one.
\item It extends to a homotopy $[0, 1] \times \hat{D} \to \CP{1}$ which is a topological embedding at each time.
\item Its extension deforms cross ratios Lipschitz-continuously, with Lipschitz constant $\arctanh M$.
\end{itemize}
To state the last property precisely, pick four distinct points in $\hat{D}$, and let $\rho_a$ be the cross ratio of their images in $\CP{1}$ at time $a \in [0, 1]$. The property is that
\[ d_{\therefore}(\rho_a, \rho_b) \le |b - a| \arctanh M \]
for all $a, b \in [0, 1]$, no matter which four points we picked.
\end{lem}
\begin{proof}
Identify $\hat{D}$ with the closed upper half-plane in $\CP{1}$. Find a conformal map $\delta \maps D \to \CP{1}$ with the properties listed in Corollary~\ref{gp}, using the method from the proof of the corollary. We know, from the details of the proof, that $\delta$ extends to a $\frac{1+M}{1-M}$-quasiconformal homeomorphism $\mathring{\delta} \maps \CP{1} \to \CP{1}$. By composing with a suitable M\"{o}bius transformation, we can arrange for $\mathring{\delta}$ to fix $0$, $1$, and $\infty$.

We'll now travel from the identity to $\mathring{\delta}$ along a clever path found by Dao-shing, and studied further by Gehring and Reich~\cite{parametric-rep}\cite[Lemma~3.3]{area-distortion}\cite[\S 2]{planar-qc}. Write $\mu$ for the Beltrami differential of $\mathring{\delta}$. The family of $(-1, 1)$-forms
\[ \mu_a = \frac{\mu}{|\mu|} \tanh (a \arctanh |\mu|), \]
parameterized by $a \in [0, 1]$, goes continuously from $0$ to $\mu$. (To resolve the ambiguity in the formula, set $\mu_a$ to zero wherever $\mu$ is zero.) Let $\mathring{\Delta}_a \maps \CP{1} \to \CP{1}$ be the map which has Beltrami differential $\mu_a$ and fixes $0$, $1$, and $\infty$. Using an argument by Lehto,\footnote{Write out the Beltrami differential of $\mathring{\Delta}_b \circ \mathring{\Delta}_a^{-1}$ using formula (I.4.4) of \cite{unival-teich}, and simplify using the angle addition identity for the hyperbolic tangent. You'll see that the magnitude of the Beltrami differential is bounded by $\tanh [(b - a) \arctanh M]$ almost everywhere. Taking the hyperbolic arctangent of both sides, and then writing the hyperbolic arctangent in terms of the natural logarithm, you'll get the desired bound on the dilatation of $\mathring{\Delta}_b \circ \mathring{\Delta}_a^{-1}$. Lehto uses the same kind of argument to prove Theorem~I.4.7 of \cite{unival-teich}.} you can show that $\mathring{\Delta}_b \circ \mathring{\Delta}_a^{-1}$ is $\big(\tfrac{1+M}{1-M}\big)^{|b - a|}$-quasiconformal\footnote{This formula conflicts with the second displayed equation in \cite[\S 2]{planar-qc}, which would give $e$ as the base of the exponent.} for all $a, b \in [0, 1]$. Teichm\"{u}ller's theorem on the quasi-invariance of cross ratios then bounds $d_{\therefore}(\rho_a, \rho_b)$ as stated~\cite{teich-thm}, just as in the proof of Corollary~\ref{awt}.

Each $\mu_a$ vanishes on $D$, because $\mu$ vanishes there. Each quasiconformal map $\mathring{\Delta}_a \maps \CP{1} \to \CP{1}$ thus restricts to a conformal map $\Delta_a \maps D \to \CP{1}$, whose Schwarzian derivative we'll call $\varepsilon_a$. It's apparent from the formula for $\mu_a$ that $|\mu_a| \le |\mu|$, and we know from the proof of Corollary~\ref{awt} that $|\mu| < M$ almost everywhere. This implies that $|\varepsilon_a| < 6M\eta^2$, as proven by K\"{u}hnau and by Lehto~\cite[\S II.3.3]{unival-teich}. Hence, by Corollary~\ref{gp}, $\mathring{\Delta}_a$ restricts to a topological embedding $\hat{D} \to \CP{1}$.

Our bound on the deformation of cross ratios ensures that $\mathring{\Delta}_a$ varies continuously with respect to $a$. Hence, $\mathring{\Delta}$ is a homotopy from the identity to $\mathring{\delta}$. Its restriction $\Delta$ is the homotopy we want.
\end{proof}
\begin{proof}[Proof of Theorem~\ref{bicorn-quant-perturb} for $|\Im r| \le \pi$]
Consider two points $\mathfrak{m}_\text{start}, \mathfrak{m}_\text{end} \in \hat{B} \smallsetminus \{\pm\infty\}$ separated by a vertical distance of $\pi$ or less. Pick an embedding $\hat{\gamma} \maps \hat{B} \to \hat{U}$ of the kind produced by Theorem~\ref{bicorn-qual-perturb}. Let $r$ be the displacement from $\mathfrak{m}_\text{start}$ to $\mathfrak{m}_\text{end}$, and let $r'$ be the displacement from $\hat{\gamma}(\mathfrak{m}_\text{start})$ to $\hat{\gamma}(\mathfrak{m}_\text{end})$. By Proposition~\ref{bicorn-deform-rigid}, $r'$ doesn't depend on our choice of $\hat{\gamma}$.

Suppose $|\varepsilon| < 2M\lambda^2$ for some $M \in (0, \tfrac{1}{3})$. For convenience, we'll declare $\tfrac{1+M}{1-M}$ as a new parameter $K \in (1, \infty)$, which varies monotonically with $M$. We want to show that
\[ |r' - r| \le (K-1)(|\Re r| + \tfrac{3}{2}\pi). \]
It's easy to see that both $r$ and $r'$ are zero when $\mathfrak{m}_\text{start}$ and $\mathfrak{m}_\text{end}$ are the same, so let's restrict ourselves to the case where $\mathfrak{m}_\text{start}$ and $\mathfrak{m}_\text{end}$ are distinct.

Because $\mathfrak{m}_\text{start}$ and $\mathfrak{m}_\text{end}$ are separated by a vertical distance of $\pi$ or less, we can find a bicorn disk $D \subset B$ whose circular closure $\hat{D}$ contains both of them. Pick a homotopy $\Delta \maps [0, 1] \times D \to \CP{1}$ with the properties listed in Lemma~\ref{hty-awt}. Let $\hat{\Delta}$ be its extension to $[0, 1] \times \hat{D}$. By composing with a suitable time-dependent M\"{o}bius transformation, we can arrange for $\hat{\Delta}$ to always send $-\infty$ and $+\infty$ to $0$ and $\infty$, respectively. Then we can lift it along the branched covering $\hat{U} \to \CP{1}$ to a homotopy $\hat{\Gamma} \maps [0, 1] \times \hat{D} \to \hat{U}$ which is a topological emedding at each time.

For each time $a \in [0, 1]$, let $\rho_a$ be the cross ratio
\[ (0, \hat{\Delta}_a(\mathfrak{m}_\text{start}), \hat{\Delta}_a(\mathfrak{m}_\text{end}), \infty) = \left. \frac{\hat{\Delta}_a(\mathfrak{m}_\text{end}) - 0}{\hat{\Delta}_a(\mathfrak{m}_\text{start}) - 0} \middle/ \frac{\hat{\Delta}_a(\mathfrak{m}_\text{end}) - \infty}{\hat{\Delta}_a(\mathfrak{m}_\text{start}) - \infty} \right., \]
which simplifies to $\hat{\Delta}_a(\mathfrak{m}_\text{end})/\hat{\Delta}_a(\mathfrak{m}_\text{start})$. Let $r_a$ be the displacement from $\hat{\Gamma}_a(\mathfrak{m}_\text{start})$ to $\hat{\Gamma}_a(\mathfrak{m}_\text{end})$, and observe that $\exp r_a = \rho_a$. In other words, $r_a$ is a lift of the path $\rho_a$ along the covering map
\begin{center}
\begin{tikzcd}
\C \smallsetminus 2\pi i\Z \arrow[r, "\exp"] & \CP{1} \smallsetminus \{0, 1, \infty\}.
\end{tikzcd}
\end{center}

Write the hyperbolic metric on $\CP{1} \smallsetminus \{0, 1, \infty\}$ and the pullback of that metric to $\C \smallsetminus 2\pi i\Z$ as $\eta_\therefore^2$ and $\eta_{\ldots}^2$, respectively. Because the homotopy $\hat{\Delta}$ was produced by Lemma~\ref{hty-awt}, the path $\rho_a$ is Lipschitz with respect to $\eta_\therefore^2$, with Lipschitz constant $\arctanh M$. The lift $r_a$ is therefore Lipschitz with respect to $\eta_{\ldots}^2$, with the same Lipschitz constant. In particular,
\[ d_{\ldots}(r_0, r_1) \le \arctanh M, \]
where $d_{\ldots}$ is the distance induced by the metric $\eta_{\ldots}^2$.

The homotopy $\hat{\Gamma}$ restricts to a homotopy to $\Gamma \maps [0, 1] \times D \to U$ which is conformal at each time. The initial map $\Gamma_0$ is a M\"{o}bius map, so $\hat{\Gamma}_0$ is a translation map, by Proposition~\ref{bicorn-region-rigid}. Hence, $r_0 = r$. The final map $\Gamma_1$ has Schwarzian derivative $\varepsilon$, so the embedding $\hat{\gamma} \maps \hat{B} \to \hat{U}$ that we chose at the beginning matches $\hat{\Gamma}_1$ up to a translation, thanks to Proposition~\ref{bicorn-deform-rigid}. Hence, $r_1 = r'$. The bound above now says that
\[ d_{\ldots}(r, r') \le \arctanh M. \]

To get the weaker but more explicit bound stated in the theorem, we use the inequality~\cite[Theorem 14.3.1]{local-viewpoint}
\[ \frac{|dz|}{2|z|\big(\big|\log |z|\big| + \tfrac{3}{2}\pi\big)} \le \eta_\therefore, \]
which pulls back along $\exp$ to the inequality
\[ \frac{|dz|}{2(|\Re z| + \tfrac{3}{2}\pi)} \le \eta_{\ldots}. \]
The left-hand side, which I'll call $\eta_\text{tract}$, has a nice hyperbolic interpretation. Under the metric $\eta_\text{tract}^2$, the complex plane becomes the universal cover of a tractricoid---a pair of hyperbolic cusps glued together along a horocycle. You can see this concretely by identifying the closed left and right half-planes in $(\C, \eta_\text{tract}^2)$ with the horodisks $\Re z \le -\tfrac{3}{2}\pi$ and $\tfrac{3}{2}\pi \le \Re z$ in the left and right half-plane models of $\mathbb{H}$.

We've shown that $r'$ is in the closed $\eta_{\ldots}^2$ ball around $r$ of radius $\arctanh M$, which can also be expressed as $\tfrac{1}{2} \log K$. We know from the inequality above that this $\eta_{\ldots}^2$ ball is inside the closed $\eta_\text{tract}^2$ ball around $r$ of the same radius. Using the hyperbolic interpretation, you can show that this $\eta_\text{tract}^2$ ball is inside the closed $|dz|^2$ ball around $r$ of radius $(K-1)(|\Re r| + \tfrac{3}{2}\pi)$. This bounds $|r' - r|$ as desired.
\end{proof}
\begin{proof}[Proof of Theorem~\ref{bicorn-quant-perturb} in general]
Consider any two points $\mathfrak{m}_\text{start}, \mathfrak{m}_\text{end} \in \hat{B} \smallsetminus \{\pm\infty\}$. Pick an embedding $\hat{\gamma} \maps \hat{B} \to \hat{U}$ of the kind produced by Theorem~\ref{bicorn-qual-perturb}, and define $r$ and $r'$ like we did in the proof of the base case. For convenience, let's restrict our attention to the case where $\Im r \ge 0$. The argument we'll use adapts straightforwardly to the other case.

Starting at $\mathfrak{m}_\text{start}$, walk upward in steps of length $\pi$ until your vertical distance to $\mathfrak{m}_\text{end}$ is $\pi$ or less. This produces a sequence of points $\mathfrak{m}_1, \ldots, \mathfrak{m}_j$, with $\mathfrak{m}_1 = \mathfrak{m}_\text{start}$. Notice that $|\Im r| \in [0, j\pi]$.

By construction, the displacement from each point $\mathfrak{m}_\nu$ to $\mathfrak{m}_{\nu+1}$ is $i\pi$. Let $r'_\nu$ be the displacement from $\hat{\gamma}(\mathfrak{m}_\nu)$ to $\hat{\gamma}(\mathfrak{m}_{\nu+1})$. Let $r_\text{fin}$ be the displacement from $\mathfrak{m}_j$ to $\mathfrak{m}_\text{end}$, and let $r'_\text{fin}$ be the displacement from $\hat{\gamma}(\mathfrak{m}_j)$ to $\hat{\gamma}(\mathfrak{m}_\text{end})$.

Suppose $|\varepsilon| < 2M\lambda^2$ for some $M \in (0, \tfrac{1}{3})$. As before, we'll declare $\tfrac{1+M}{1-M}$ as a new parameter $K \in (1, \infty)$. Since the vertical distance from each point $\mathfrak{m}_\nu$ to $\mathfrak{m}_{\nu+1}$ is $\pi$, we can apply the base case of the theorem and conclude that
\begin{align*}
|r'_\nu - i\pi| & \le (K-1)(|\Re i\pi| + \tfrac{3}{2}\pi) \\
& = (K-1)\tfrac{3}{2}\pi.
\end{align*}
Applying the base case to the step from $\mathfrak{m}_j$ to $\mathfrak{m}_\text{end}$, we learn that
\[ |r'_\text{fin} - r_\text{fin}| \le (K-1)(|\Re r_\text{fin}| + \tfrac{3}{2}\pi). \]

Taking the difference between the expressions
\begin{align*}
r & = i\pi + \ldots + i\pi + r_\text{fin} \\
r' & = r'_1 + \ldots + r'_{j-1} + r'_\text{fin}
\end{align*}
and applying the triangle inequality, we get the bound
\[ |r' - r| \le |r'_1 - i\pi| + \ldots + |r'_{j-1} - i\pi| + |r'_\text{fin} - r_\text{fin}|. \]
Combining this with the step-by-step bounds from above, we learn that
\begin{align*}
|r' - r| & \le (K-1)\tfrac{3}{2}(j-1)\pi + (K-1)(|\Re r_\text{fin}| + \tfrac{3}{2}\pi) \\
& = (K-1)(|\Re r_\text{fin}| + \tfrac{3}{2}j\pi).
\end{align*}
Noting that $\Re r_\text{fin} = \Re r$, we end up with the desired bound.
\end{proof}
\appendix
\section{Holomorphicity of generalized eigenfunctions}\label{gen-eigfns}
Physicists like to sort the stationary and steadily decaying states of a particle on a path into two classes: bound states, which are fairly localized, and unbound states, which are not. The most basic model of the particle tells us only about the bound states. They're described by the eigenfunctions of a Hamiltonian operator on the state space $L^2(I)$, where $I$ is the real interval parameterizing the path. We can elaborate this model by expanding the state space to a rigged Hilbert space $\Phi \subset L^2(I) \subset \Phi^\times$. The elements of $\Phi^\times$ can be seen as generalized functions, on which the Hamiltonian acts naturally. The generalized eigenfunctions of the Hamiltonian describe the unbound states.

In the setting of Section~\ref{context}, the Hamiltonian also acts on holomorphic functions over a neighborhood $\Omega$ of $I$. Let's suppose $\Omega$ is simply connected. For good choices of rigged Hilbert space, the generalized eigenfunctions of the Hamiltonian correspond one-to-one with the holomorphic eigenfunctions that satisfy some regularity condition.

% all solutions are classical solutions. see also...
% - D.H. Griffel, _Applied Functional Analysis_, end of Section 2.2 (pp. 40--41).
% https://mathoverflow.net/a/39842/1096
% https://math.stackexchange.com/a/460809/16063
% https://math.stackexchange.com/a/1823206/16063
As an example, take $\Phi$ to be the space of compactly supported smooth functions on $I$, so $\Phi^\times$ is the complex conjugate of the standard space of distributions. Using the method of integrating factors, as described in \cite{weak-solutions}, we see that each generalized eigenfunction of the Hamiltonian can be represented by an analytic function on $I$. Since the potential extends to a holomorphic function over $\Omega$, the integration of the eigenfunction equation can be continued along any path through $\Omega$. Because $\Omega$ is simply connected, the result is a holomorphic eigenfunction defined consistently over all of $\Omega$. In this way, each generalized eigenfunction can be represented by a holomorphic eigenfunction. Conversely, every holomorphic eigenfunction represents a generalized eigenfunction, with no regularity condition needed.

For another example, take $\Phi$ to be the space of Schwartz functions on $I$, so $\Phi^\times$ is the complex conjugate of the space of tempered distributions. Since every tempered distribution is a standard distribution, we already know from the argument above that each generalized eigenfunction of the Hamiltonian can be represented by a holomorphic eigenfunction. Conversely, a holomorphic eigenfunction represents a generalized eigenfunction as long as its restriction to $I$ can be bounded by a polynomial.

It would be interesting to know if there are examples where $I$ is the positive real axis, $\Omega$ is the wedge $e^{i(-\alpha, \alpha)} I$ for some angle $\alpha$, and $\Phi$ is the space of dilation-analytic functions on $\Omega$~\cite{analytic-spec-theory-ii}. A holomorphic eigenfunction might represent a generalized eigenfunction as long as its Mellin transform doesn't grow too fast~\cite{analytic-partial-ip}.
\section{Kummer's transformation}\label{kummer-tfm}
To find a local Liouville equivalence relating any two Hill's operators, it's enough to find a Liouville equivalence relating each operator to $(\frac{\partial}{\partial z})^2$. That's what we'll do in this appendix, using a transformation first studied by Kummer~\cite[\S 2]{old-and-new}.

Take a Hill's operator
\[ \mathcal{H} = \left(\frac{\partial}{\partial z}\right)^2 - \frac{q}{2} \]
on a region $\Omega \subset \C$, and pick any point $\mathfrak{m} \in \Omega$. We aim to find a neighborhood $\Omega''$ of $\mathfrak{m}$ and a conformal map $y \maps \Omega'' \to \C$ with the property that
\[ \mathcal{H} = y_z^{3/2} \circ \left(\frac{\partial}{\partial y}\right)^2 \circ y_z^{1/2}. \]
Our first step is to solve the equation $\mathcal{H}\psi = 0$, choosing a solution that doesn't vanish at $\mathfrak{m}$. On the neighborhood $\Omega'$ of $\mathfrak{m}$ where $\psi$ stays non-zero,
\begin{align*}
\mathcal{H} & = \left(\frac{\partial}{\partial z} + \frac{\psi_z}{\psi}\right)\left(\frac{\partial}{\partial z} - \frac{\psi_z}{\psi}\right) \\
& = \left(\psi^{-1} \circ \frac{\partial}{\partial z} \circ \psi\right) \left(\psi \circ \frac{\partial}{\partial z} \circ \psi^{-1}\right) \\
& = \psi^{-3} \circ \left(\psi^2 \circ \frac{\partial}{\partial z}\right) \left(\psi^2 \circ \frac{\partial}{\partial z}\right) \circ \psi^{-1}.
\end{align*}
Putting $\Omega'' \subset \Omega'$, we see that any holomorphic map $y \maps \Omega'' \to \C$ with $dy^{-1/2} = \psi\,dz^{-1/2}$ will have the property we want.

Choose $\Omega''$ to be simply connected. That ensures we can solve $\mathcal{H} \tilde{\psi} = 0$ on $\Omega''$ with $\on{Wr}(\psi, \tilde{\psi}) = 1$. The map $y = \tilde{\psi}/\psi$ has the property we want.

\bibliographystyle{utphys}
\bibliography{free-enough}
\end{document}